\documentclass[aop,preprint]{imsart}

\RequirePackage[OT1]{fontenc}
\RequirePackage{amsthm,amsmath,natbib}
\RequirePackage[colorlinks]{hyperref}
\RequirePackage{hypernat}

\RequirePackage[OT1]{fontenc}
\RequirePackage{amsthm,amsmath,natbib}
\RequirePackage[colorlinks]{hyperref}
\RequirePackage{hypernat}

\numberwithin{equation}{section}

\usepackage{amssymb}
\usepackage{comment}
\usepackage{appendix}
\usepackage{graphicx}
\usepackage{subfigure}
\usepackage{enumerate}
\usepackage{comment}
\usepackage{pdfsync}
\input xy
\xyoption{all}

\startlocaldefs

\newtheorem{theorem}{Theorem}[section]
\newtheorem{lemma}[theorem]{Lemma}

\newtheorem{corollary}[theorem]{Corollary}

\newtheorem{definition}[theorem]{Definition}

\newtheorem{remark}[theorem]{Remark}

\newcommand{\B}{\mathbf{B}}

\newcommand{\Ev}{\mathbf{E}}
\newcommand{\E}{\mathbf{E}}

\newcommand{\G}{\mathbf{G}}
\newcommand{\N}{\mathbf{N}}
\newcommand{\Z}{\mathbf{Z}}

\newcommand{\Pv}{\mathbf{P}}

\newcommand{\CG}{G}

\newcommand{\CX}{\mathcal {X}}

\newcommand{\Hh}{H}

\newcommand*{\pid}{\pi^{\diamond}}
\newcommand*{\td}{t^{\diamond}}
\newcommand*{\tstop}{t_{\text{stop}}}

\newcommand*{\la}{\lambda}

\newcommand*{\ve}{\varepsilon}

\newcommand*{\al}{\alpha}

\newcommand*{\Rb}{\mathbb R}

\newcommand*{\un}[1]{\underline{#1}}

\newcommand*{\be}{\begin{equation}}
\newcommand*{\ee}{\end{equation}}
\newcommand*{\ba}{\begin{aligned}}
\newcommand*{\ea}{\end{aligned}}
\newcommand*{\barr}{\begin{array}{c}}
\newcommand*{\earr}{\end{array}}
\newcommand*{\Vv}{{\text{\bf Var}}}

\newcommand*{\fl}[1]{\lfloor{#1}\rfloor}

\newcommand*{\ind}{\mathbf{1}}
\newcommand*{\trel}{t_{{\rm rel}}}
\newcommand*{\tu}{t_{{\rm u}}}
\newcommand*{\tcov}{t_{{\rm cov}}}
\newcommand*{\thit}{t_{{\rm hit}}}
\newcommand*{\tmix}{t_{{\rm mix}}}

\newcommand*{\TV}{{\rm TV}}

\newcommand*{\tqu}[1]{t^{\text{quant}}_{#1}}

\endlocaldefs

\begin{document}
\begin{frontmatter}
\title{Mixing and relaxation time for Random Walk on Wreath Product Graphs}
\runtitle{Mixing time for walks on wreath products}

\begin{aug}

\runauthor{J\'ulia Komj\'athy and Yuval Peres}

\author{J\'ulia Komj\'athy \thanks{J. Komj\'athy was supported by the grant
 KTIA-OTKA  $\#$ CNK 77778, funded by the Hungarian National Development Agency (NFÜ)  from a
source provided by KTIA.}}
\address{Budapest University of Technology and Economics,
\tt{komyju@math.bme.hu}}

\author{Yuval Peres}
\address{Microsoft Research
\tt{peres@microsoft.com}}

\affiliation{Budapest University of Technology and Economics and Microsoft Research}

\end{aug}
\date{\today}

\begin{abstract}
Suppose that $\CG$ and $\Hh$ are finite, connected graphs, $\CG$ regular, $X$ is a lazy random walk on $\CG$ and $Z$ is a reversible ergodic Markov chain on $\Hh$.  The generalized lamplighter chain $X^\diamond$ associated with $X$ and $Z$ is the random walk on the wreath product $\Hh \wr \CG$, the graph whose vertices consist of pairs $(\un f,x)$ where $\un f=\left(f_v\right)_{v\in V(\CG)}$ is a labeling of the vertices of $\CG$ by elements of $\Hh$ and $x$ is a vertex in $\CG$. In each step, $X^\diamond$ moves from a configuration $(\un f,x)$ by updating $x$ to $y$ using the transition rule of $X$ and then independently updating both $ f_x$ and $f_y$ according to the transition probabilities on $\Hh$; $f_z$ for $z \neq x,y$ remains unchanged.  We estimate the mixing time of $X^\diamond$ in terms of the parameters of $\Hh$ and $\CG$. Further, we show that the relaxation time of $X^\diamond$ is the same order as the maximal expected hitting time of $G$ plus $|G|$ times the relaxation time of the chain on $H$.
\end{abstract}

\begin{keyword}[class=AMS]
\kwd[Primary ]{60J10} 
\kwd{60D05}
\kwd{37A25}
\end{keyword}

\begin{keyword}
\kwd{Random walk, generalized lamplighter walk, wreath product, mixing time, relaxation time.}
\end{keyword}

\end{frontmatter}

\maketitle

\section{Introduction}

Suppose that $\CG$ and $\Hh$ are finite connected graphs with vertices $V(\CG)$, $V(\Hh)$ and edges $E(\CG)$, $E(\Hh)$, respectively.  We refer to $\CG$ as the base and $\Hh$ as the lamp graph, respectively. Let $\CX(\CG) = \{\un f \colon V(\CG) \to \Hh\}$ be the set of markings of $V(\CG)$ by elements of $\Hh$.
 The wreath product $\Hh \wr \CG$ is the graph whose vertices are pairs $(\un f,x)$
where $\un f=\left(f_v\right)_{v\in V(\CG)} \in \CX(\CG)$ and $x \in V(\CG)$.
There is an edge between $(\un f,x)$ and $(\un g,y)$ if and only if $(x,y) \in E(\CG)$, $\left(f_x,g_x\right), \left(f_y,g_y\right) \in E(\Hh)$ and $f_z = g_z$ for all $z \notin \{x,y\}$.  Suppose that $P$ and $Q$ are transition matrices for Markov chains on $\CG$ and on $\Hh$, respectively.  The generalized lamplighter walk $X^\diamond$ (with respect to the transition matrices $P$ and $Q$) is the Markov chain on $\Hh\wr \CG$ which moves from a configuration $(\un f,x)$ by
\begin{enumerate}
\item picking $y$ adjacent to $x$ in $\CG$ according to $P$, then
\item updating each of the values of $f_x$ and $f_y$ independently according to $Q$ on $\Hh$.
\end{enumerate}
The state of lamps $f_z$ at all other vertices $z\in \CG$ remain fixed.  It is easy to see that if $P$ and  $Q$ are irreducible, aperiodic and reversible with stationary distribution $\pi_\CG$ and $\pi_\Hh$, respectively, then the unique stationary distribution of $X^\diamond$ is the product measure
\[ \pid\big((\un f,x)\big)= \pi_\CG(x) \cdot\!\!\!\prod_{v\in V(\CG)} \pi_\Hh\left(f_v\right),\]
and $X^\diamond$ is itself reversible.  In this article, we will be concerned with the special case that $P$ is the transition matrix for the \emph{lazy random walk} on $\CG$.  In particular, $P$ is given by
\begin{equation}
\label{eq::lazy_rw_definition}
P(x,y) := \begin{cases} \frac{1}{2} \text{ if } x = y,\\ \frac{1}{2d(x)} \text{ if } \{x,y\} \in E(\CG), \end{cases}
\end{equation}
for $x,y \in V(\CG)$ and where $d(x)$ is the degree of $x$.
We further assume that the transition matrix $Q$ on $\Hh$ is irreducible and aperiodic. This and the assumption \eqref{eq::lazy_rw_definition} guarantees that we avoid issues of periodicity.
\begin{figure}[ht]
\includegraphics[height=6cm]{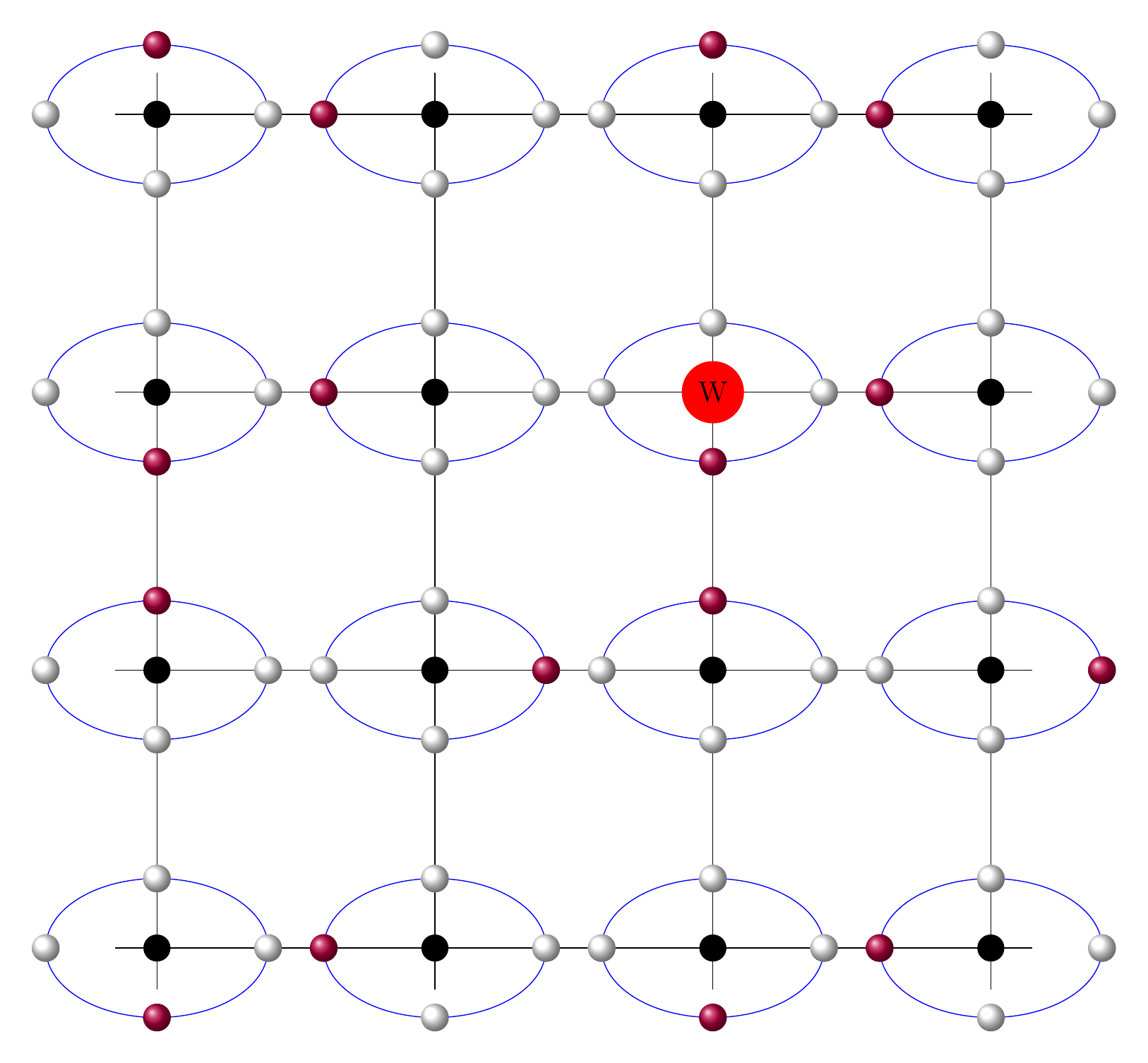}
\caption{A typical state of the generalized lamplighter walk.
Here $\Hh=\Z_4$ and $\CG=\Z_4^2$, the red bullets on each copy of $\Hh$ represents the state of the lamps over each vertex $v \in \CG$ and the walker is drawn as a red $W$ bullet.}
\end{figure}

\subsection{Main Results}

In order to state our general result, we first need to review some basic terminology from the theory of Markov chains.
Let $P$ be the transition kernel for a lazy random walk on a finite, connected graph $\CG$ with stationary distribution $\pi$.

The $\ve$-\emph{mixing time} of $P$ on $\CG$ in total variation distance is given by
\begin{equation}
\label{eqn::tmix_definition}
 \tmix(\CG,\ve) := \min\left\{ t \geq 0: \max_{x\in V(\CG)}\frac12 \sum_{y} \left|P^t(x,y)-\pi(y)\right| \leq \ve \right\}.
\end{equation}
Throughout, we set $\tmix(\CG) := \tmix(\CG, \tfrac{1}{4})$.

  The \emph{relaxation time} of a reversible Markov Chain with transition matrix $P$ is
\begin{equation}
\label{eqn::trel_definition}
 \trel(\CG) := \frac{1}{1-\lambda_2}
\end{equation}
where $\lambda_2$ is the second largest eigenvalue of $P$.

The \emph{maximal hitting time} of $P$ is
\begin{equation}
\label{eqn::thit_definition}
 \thit(\CG) := \max_{x,y\in V(\CG)} \E_x[\tau_y],
\end{equation}
where $\tau_y$ denotes the first time $t$ that $X(t) = y$ and $\E_x$ stands for the expectation under the law in which $X(0) = x$.
The random cover time $\tau_{\text{cov}}$ is the first time when all vertices have been visited by the walker $X$, and the cover time $\tcov(G)$ is
\begin{equation}
\label{eqn::tcov_definition}
 \tcov(\CG) := \max_{x\in V(\CG)} \E_x[\tau_{\text{cov}}].
\end{equation}
The next needed concept is that of strong stationary times.
\begin{definition}\label{def::sst}
A randomized stopping time $\tau$ is called a strong stationary time for the Markov chain $X_t$ on $G$ if
\[ \Pv_{x}\left[ X_{\tau}=y, \tau = t \right] = \pi(y) \Pv_x[\tau=t], \]
that is, the position of the walk when it stops at $\tau$ is independent of the value of $\tau$.
\end{definition}
\noindent The adjective randomized means that the stopping time can depend on some extra randomness, not just purely the trajectories of the Markov chain, for a precise definition see \cite[Section 6.2.2]{LPW08}.
\begin{definition}\label{def::halting_state} A state $h(x)\in V(G)$ is called a  halting state for a stopping time $\tau$ and initial state $x$ if $\{X_t=h(x)\}$ implies $\{\tau\le t\}$.
\end{definition}
Our main results are summarized in the following theorems:

\begin{theorem}
\label{thm::main_relax}
Let us assume that $\CG$ and $\Hh$ are connected graphs with $\CG$ regular and the Markov chain on $\Hh$ is ergodic and reversible.
Then there exist universal constants $c_1,C_1$ such that the relaxation time of the  generalized lamplighter walk on
$\Hh\wr \CG$ satisfies
\begin{align}
 c_1  \leq \frac{\trel(\Hh\wr \CG)}{  \thit(\CG) +|\CG|\trel(\Hh)}  \leq C_1, \label{eqn::trel_main}
\end{align}
\end{theorem}

\begin{theorem}\label{thm::main_mixing}
Assume that the conditions of Theorem \ref{thm::main_relax} hold and further assume that the chain with transition matrix $Q$ on $H$ is lazy, i.e. $Q(x,x)\ge \frac12 \ \forall x \in H$.
Then there exist universal constants $c_2,C_2$ such that the mixing time of the  generalized lamplighter walk on $\Hh\wr \CG$ satisfies
\be\label{eqn::tmix_main}\begin{aligned}
 &c_2 \left( \tcov(G)+\trel(H) |G| \log |G| + |G|\tmix(H)\right) \leq \tmix(\Hh\wr \CG), \\
 &\tmix(H\wr G) \leq C_2 \left( \tcov(\CG) +|\CG|\tmix(\Hh,\frac{1}{|\CG|})\right).
\end{aligned}
\ee
If further the Markov chain is such that
\begin{description}
\item[(A)] \label{ass::sst_exist} There is a strong stationary time $\tau_H$ for the Markov chain on $H$ which possesses a halting state $h(x)$ for every initial starting point $x\in H$,
\end{description}
then the upper bound of \ref{eqn::tmix_main} is sharp.
\end{theorem}
\begin{remark}
The laziness assumption on the transition matrix $Q$ on $H$ is only used to get the term $c_2 |G| \tmix(H)$ in \eqref{eqn::tmix_main}. All the other bounds hold without the laziness assumption.
\end{remark}

\begin{remark}
If the Markov Chain on $H$ is such that \[ \tmix(H, \ve) \le \tmix(H, 1/4) + \trel(H) \log \ve,\] then the upper bound matches the lower bound. This holds for many natural chains such as lazy random walk on hypercube $\Z_2^d$, tori $\Z_n^d$, some walks on the permutation group $S_n$ (the  random transpositions or random adjacent transpositions shuffle, and the top-to-random shuffle, for instance).
\end{remark}

\begin{remark}
Many examples where Assumption ({\bf A}) holds are given in the thesis of Pak \cite{PAK97}, including the cycle $\Z_n$, the hypercube $\Z_2^d$ and more generally tori $\Z_n^d, n, d \in \N$ and dihedral groups $\Z_2\ltimes \Z_n, n\in \N$ are also obtained by the construction of strong stationary times with halting states on direct and semidirect product of groups. Further, Pak constructs strong stationary times possessing halting states for the random walk on $k$-sets of $n$-sets, i.e. on the group $S_n/(S_k\times S_{n-k})$, and on subsets of $n\times n$ matrices over the full linear group, i.e. on $GL(n, \mathbb F_q)/(GL(k, \mathbb F_q)\times GL(n-k, \mathbb F_q))$.
\end{remark}

\subsection{Previous Work}
The mixing time of $\Z_2 \wr \CG$ was first studied by H\"aggstr\"om and Jonasson in \cite{HJ97} in the case of $\CG$ being the complete graph $K_n$ and the one-dimensional cycle $\Z_n$. Generalizing their results, Peres and Revelle \cite[Theorem 1.2, 1.3]{PR04} proved that there exists constants $c_i, C_i$ depending on $\ve$ such that for any transitive graph $\CG$,
\[ \ba c_1\thit(\CG)  &\leq \trel(\Z_2 \wr \CG) \leq C_1  \thit (\CG),\\
c_2 \tcov(\CG)  &\leq \tmix(\Z_2 \wr \CG, \ve) \leq C_2  \tcov (\CG).
\ea \]
The vertex transitivity condition was dropped in \cite[Theorem 19.1, 19.2]{LPW08}.
These bounds match with Theorems \ref{thm::main_relax} and \ref{thm::main_mixing} since $\Hh_n = \Z_2$ implies that the terms not containing $\Hh_n$ in the denominator of \eqref{eqn::trel_main} and in the bounds in \eqref{eqn::tmix_main} dominate.

In \cite{MP11}, it is shown that $\tmix(\Z_2 \wr \CG_n) \sim \tfrac{1}{2} \tcov(\CG_n)$ whenever $(\CG_n)$ is a sequence of graphs satisfying some uniform local transience assumptions, including $G_n=\Z_n^d$ with $d\geq 3$ fixed.

Moving towards larger lamp spaces, if the base is the complete graph $K_n$ and $|\Hh_n|= o(n)$ one can determine the order of mixing time from \cite[Theorem 20.7]{LPW08}, since in this case the lamplighter chain is a product chain on $\prod_{i=1}^{n} \Hh_n$. Levi \cite{NL06} investigated random walks on wreath products when $\Hh\neq \Z_2$. In particular, he determined the order of the mixing time of $K_{n^{\la}}\wr K_n$, $0\le \la \le 1$, and he also had upper and lower bounds for the case $H_d\wr \Z_n$, i.e. $\Hh$ is the $d$-dimensional hypercube and the base is a cycle of length $n$, however, the bounds failed to match for general $d$ and $n$. Further, Fill and Schoolfield \cite{FS01} investigated the total variation and $l_2$ mixing time of $K_n\wr S_n$, where the base graph is the Cayley graph of the symmetric group $S_n$ with transpositions chosen as the generator set, and the stationary distribution on $K_n$ is not necessarily uniform. 

 The mixing time of $\Hh_n= \Z_2$ is closely related to the cover time of the base graph, and thus it helps understanding the geometric structure of the last visited points by random walk \cite{DPRZ_LATE06, DPRZ_COV04, BH91, MP11}.
 Further, larger lamp graphs give more information on the local time structure of the base graph $\CG$. This relates our work to the literature on blanket time (when all the local times of vertices are within a constant factor of each other) \cite{BDNP11, DLP11, WZ96}.

\subsection{Outline}
The remainder of this article is structured as follows. In Section \ref{sec::prelim} we state a few necessary theorems and lemmas about the Dirichlet form, strong stationary times, different notions of distances and their relations. In Lemmas \ref{lem::opt_sst_lamplighter} and \ref{lem::opt_sst_lamplighter_2} we construct a crucial stopping time $\tau^\diamond$ and a strong stationary time $\tau^\diamond_2$ on $\Hh \wr \CG$ which we will use several times throughout the proofs later. Then we prove the main theorem about the relaxation time in Section \ref{sec::trel}, and the mixing time bounds in Section \ref{sec::tmix}.

\section{Notations}
Throughout the paper, objects related to the base or the lamp graph will be indexed by $G$ and $H$, respectively, and $\diamond$ always refers to an object related to the whole $H\wr G$. Unless misleading, $G$ and $H$ refers also to the vertex set of the graphs, i.e. $v\in G$ means $v\in V(G).$ $\Pv_\mu, \Ev_\mu$ denotes probability and expectation under the conditional law where the initial distribution of the Markov chain under investigation is $\mu$. Similarly, $\Pv_x$ is the law under which the chain starts at $x$.

\section{Preliminaries}\label{sec::prelim}
In this section we collect the preliminary lemmas to be able to carry through the proofs quickly afterwards.
The reader familiar with notions of strong stationary times, separation distance, and Dirichlet forms might want jump forward to Lemmas \ref{lem::opt_sst_lamplighter} and \ref{lem::opt_sst_lamplighter_2} immediately, and check the other lemmas here only when needed.

The first lemma is a common useful tool to prove lower bounds for relaxation times, by giving the variational characterization of the spectral gap. First we start with a definition.

Let $P$ be a reversible transition matrix with stationary distribution $\pi$ on the state space $\Omega$ and let $\Ev_\pi [\phi]:= \sum_{y\in \Omega}\phi(y) \pi(y)$. The Dirichlet form associated to the pair $(P,\pi)$ is defined for functions $\phi$ and $\eta$ on $\Omega$  by
\[ \mathcal E(\phi, \eta):=\langle (I-P)\phi, \eta \rangle_\pi= \sum_{y\in \Omega} (I-P)\phi(y) \eta(y)\pi(y). \]
It is not hard to see \cite[Lemma 13.11]{LPW08} that
\begin{equation} \label{def::dirichlet} \mathcal E(\phi):=\mathcal E(\phi, \phi)=\frac12 \E_\pi\left[(\phi(X_1)-\phi(X_0))^2\right]\end{equation}
The next lemma relates the spectral gap of the chain to the Dirichlet form (for a short proof see \cite{AF02} or \cite[Lemma 13.12]{LPW08}):
\begin{lemma}[Variational characterization of the spectral gap]\label{lem::variational}
The spectral gap $\gamma= 1-\la_2$ of a reversible Markov Chain satisfies
\be\label{eq::variational} \gamma = \min\limits_{\phi: \Vv_{\pi} \phi\ne 0} \frac{\mathcal E[\phi]}{\Vv_\pi \phi}, \ee
where $\Vv_\pi\phi= \Ev_\pi[\phi^2]-\left(\Ev_\pi[\phi]\right)^2.$
\end{lemma}

A very useful object to prove the upper bound on $\trel$ and both bounds for $\tmix$ is the concept of strong stationary times.
Recall the definition from \eqref{def::sst}. It is not hard to see (\cite[Lemma 6.9]{LPW08}) that this is equivalent to
\be\label{eq::sst_ineq} \Pv_{x}\left[ X_t=y, \tau \le t \right] = \pi(y) \Pv_x[\tau\le t].\ee
To be able to relate the tail of the strong stationary times to the mixing time of the graphs, we need another distance from stationary measure, called the separation distance:
\be\label{eq::sep}
s_x(t):=\max_{y\in \Omega} \left[ 1-\frac{P^t(x,y)}{\pi(y)}\right].
\ee
 The relation between the separation distance and any strong stationary  time $\tau$ is the following inequality from \cite{AF02} or \cite{DF90} or \cite[Lemma 6.11]{LPW08}:
\be\label{eq::sep_ineq}\forall x \in \Omega:  s_x(t)\le \Pv_x(\tau>t). \ee
Throughout the paper, we will need a slightly stronger result than \eqref{eq::sep_ineq}. Namely, by \cite[Remark 3.39]{DF90} or from the proof of \eqref{eq::sep_ineq} in \cite[Lemma 6.11]{LPW08} it follows that in \eqref{eq::sep_ineq} equality holds \emph{if $\tau$ has a halting state $h(x)$ for $x$}.
Unfortunately, we just point out that the \cite[Remark 6.12]{LPW08} is not true and the statement can not be reversed: the state $h(x,t)$ maximizing the separation distance at time $t$ can also depend on $t$ and thus the existence of a halting state is not necessarily needed to get equality in \eqref{eq::sep_ineq}.

On the other hand, one can always construct $\tau$ such that \eqref{eq::sep_ineq} holds with equality for every $x\in \Omega$. This is a key ingredient to our proofs, so we cite it as a Theorem (with adjusted notation to the present paper).
\begin{theorem}\label{thm::AD}[Aldous, Diaconis]\emph{\cite[Proposition 3.2]{AD86}}
Let $(X_t, t\ge 0)$ be an irreducible aperiodic Markov chain on a finite state space $\Omega$ with initial state $x$ and stationary distribution $\pi$, and let $s_x(t)$ be the separation distance defined as in \eqref{eq::sep}.
Then
\begin{enumerate}
\item if $\tau$ is a strong stationary time for $X_t$, then $s_x(t) \le \Pv_x(\tau>t)$ for all $t\ge 0.$
\item Conversely, there exists a strong stationary time $\tau$ such that $s_x(t)= \Pv_x(\tau>t)$ holds with equality.
    \end{enumerate}
\end{theorem}
Combining these, we will call a strong stationary time $\tau$ \emph{separation optimal} if it achieves equality in \eqref{eq::sep_ineq}. Mind that every stopping time possessing halting states is separation optimal, but the reversed statement is not necessarily true.
The next two lemmas, which we will use several times, construct two stopping times for the graph $H\wr G$. The first one will be used to lower bound the separation distance and the second one upper bounds it.

We start with introducing the notation
\be\label{def::localtime} L_v(t)=2\sum_{i=0}^{t}\ind(X_i=v)-\delta_{X_0,v}-\delta_{X_t,v}\ee for the number of moves on the lamp graph $H_v, v\in G$ by the walker up to time $t$.  Slightly abusing terminology, we call it the local time at vertex $v\in \CG$.

Let us further denote the random walk with transition matrix $Q$ on $H$ by $Z$. Since the moves on the different lamp graphs $H_v, v\in G$ are taken independently given $L_v(t), v\in G$, we can define for each $v\in G$ an independent copy of the chain $Z$, denoted by $Z_v$, running on $H_v$. Thus, the position of the lamplighter chain at time $t$ can be described as
\[ (\un F_t, X_t ) = \left( \left(Z_v(L_v(t)) \right)_{v\in G}, X_t \right)\]
Below we will use copies of a strong stationary time $\tau_H$ for each $v\in G$, meaning that $\tau_H(v)$ is defined in terms of $Z_v$, and given the local times $L_v(t)$, $\tau_H(v)$-s are independent of each other.

\begin{lemma}\label{lem::opt_sst_lamplighter}
Let $\tau_H$ be any strong stationary time for the Markov chain on $\Hh$. Take the conditionally independent copies of $\left(\tau_H(v)\right)_{v\in\CG}$ given the local times $L_v(t)$, realized on the lampgraphs $H_v$-s and define the stopping time $\tau^\diamond$ for $X^\diamond$ by
\be\label{eq::tauhg} \tau^\diamond:= \inf\left\{t: \forall v\in G: \tau_H(v)\le L_v(t)\right\}. \ee
Then, for any starting state $(\un f_0,  x_0)$ we have
\be \label{eq::tau_1_eq} \Pv_{(\un f_0,x_0)}\left[ X^\diamond_t =(\un f, x), \tau^\diamond=t \right] =  \prod_{v\in \CG}\!\! \pi_H(f_v)\cdot  \Pv_{(\un f_0,x_0)}\left[ X_t=x, \tau^\diamond= t\right]. \ee
If further $\tau_H$ has halting states then the vectors $(\un h(f_v(0)),y)$ are halting state vectors for $\tau^\diamond$ and initial state $(\un f_0, x_0)$ for every $y\in G$.
\end{lemma}
We postpone the proof and continue with a corollary of the lemma:
\begin{corollary}\label{cor::sep_lower_bound}
Let $\tau_H$ be a strong stationary time for the Markov chain on $\Hh$ which has a halting state $h(z)$ for any $z\in H$. Then define $\tau^\diamond$ as in Lemma \ref{lem::opt_sst_lamplighter}. Then for the separation distance on the lamplighter chain $H\wr G$ the following lower bound holds:
\[  s_{(\un f_0,x_0)}(t) \ge \Pv_{(\un f_0,x_0)}\left[ \tau^\diamond > t \right].\]
\end{corollary}
\begin{proof}
Observe that reaching the halting state vector $(\un h(f_v(0)), x)$ implies the event $\tau^\diamond\le t$ so we have
\be \label{eq::cor_est} 1- \frac{\Pv_{(\un f_0, x_0)}\left[ X^\diamond_t= (\un h(f_v(0)), x) \right] }{ \pi_G(x)\prod\limits_{v\in \CG}\pi_H\left(h(f_v(0))\right) } = 1- \frac{\Pv_{(\un f_0, x_0)}\left[ X^\diamond_t= (\un h(f_v(0)), x), \tau^\diamond \le t  \right] }{ \pi_G(x)\prod\limits_{v\in \CG} \pi_H\left(h(f_v(0))\right) } \ee
 Now pick a vertex $x_{x_0,t}\in G$ which minimizes $\Pv\left[ X_t=x_{x_0,t}| \tau^\diamond\le t\right]/ \pi_G(x_{x_0,t})$. This quotient is less than $1$ since both the numerator and the denominator are probability distributions on $G$. Then, using this and Lemma \ref{lem::opt_sst_lamplighter}, the right hand side of \eqref{eq::cor_est} equals
\[ 1- \frac{\Pv_{(\un f_0,x_0)}\left[ X_t = x_{x_0,t} |  \tau^\diamond \le t\right] \Pv_{(\un f_0,x_0)}[\tau^\diamond\le t] }{ \pi_G(x_{x_0,t})} \ge 1-\Pv_{(\un f_0,x_0)}\left[ \tau^\diamond \le t \right].  \]
Clearly the separation distance is larger than the left hand side of \eqref{eq::cor_est}, and the proof of the claim follows. Note that the proof only works if $\tau_H$ has a halting state and thus it is separation-optimal.
\end{proof}

\begin{proof}[Proof of Lemma \ref{lem::opt_sst_lamplighter}]

 First we show that \eqref{eq::tau_1_eq} holds using the conditional independence of $\tau_H(v)$-s given the  number of moves $L_v(t)$ on the lamp graphs $H(v), v\in \CG$. Clearly, conditioning on the trajectory of the walker $\{ X_1, \dots, X_{t-1}, X_t=x\} := X{[1,t]}$ contains the knowledge of $L_v(t)$-s as well. We will omit to note the dependence of $\Pv$ on initial state $(\un f_0, x_0)$ for notational convenience. The left hand side of condition \eqref{eq::sst_ineq} equals
\[ \Pv\left[X_t^\diamond=(\un f,x), \tau^\diamond\le t \right]
 = \!\!\!\sum_{X_{[1,t]}} \!\!\!\Pv\left[ X_{t}^\diamond=(\un f,x),\tau^\diamond\le t| X_{[1,t]} \right] \Pv\left[X_{[1,t]}\right].\]
Recall that $Z_v$ stands for the Markov chain on the lamp graph $H_v$, and their conditional independence given $L_v(t)$-s.
Due to \eqref{eq::sst_ineq} and $\tau_H$ being strong stationary for $\Hh$ we have for all $v\in\CG$ that
\[ \Pv[Z_v(L_v(t))=f_v, \tau_H(v)\le L_v(t)|X_{[1,t]}] =\pi_{\Hh}(f_v)\cdot \Pv[\tau_H(v)\le L_v(t)|X_{[1,t]}].\]
Now we use that $\tau_H(v)$-s are conditionally independent given the local times to see that
\[ \ba  \Pv&\left[ X_{t}^\diamond=(\un f,x),\tau^\diamond\le t| X_{[1,t]} \right]\\
&=\Pv\left[ \forall v\in G: Z_v(L_v(t))=f_v,  \tau_H(v)\le L_v(t), X_t=x,  | X_{[1,t]} \right] \\
&= \prod_{v\in \CG}\!\! \pi_H(f_v)\!\prod_{v\in \CG} \!\!\Pv\left[\tau_{\Hh}(v)\le L_v(t)|X_{[1,t]}\right]  \ea \]
Note that the second product gives exactly $\Pv\left[\tau^\diamond \le t | X_{[1,t]} \right]$, yielding
\be\label{eq::proof_sst} \Pv\left[X_{t}^\diamond=(\un f,x), \tau^\diamond\le t \right]
=\prod_{v\in \CG}\!\! \pi_H(f_v)\!\! \sum_{X_{[1,t]}}\!\Pv\left[\tau^\diamond \le t | X_{[1,t]} \right] \Pv[X_{[1,t]} ]\ee
As $X_t=x$ remains fixed over the summation, thus summing over all possible $X[1,t]$ trajectories yields
\[\Pv[X_t^\diamond=(\un f, x), \tau^\diamond \le t]=\prod\limits_{v\in \CG}\!\! \pi_H(f_v) \Pv[X_t=x, \tau^\diamond \le t].\]To turn the inequality $\tau^\diamond \le t$ inside the probability to equality can be done the same way as in \eqref{eq::sst_ineq} and left to the reader.
 To see that the vector of halting states $(\un h(f_v(0)),y)$ is a halting state for $\tau^\diamond$ for any $y\in G$ is based on the simple fact that reaching the halting state vector $\left( h(f_v)_{v\in \CG},y \right)$ means that all the halting states $h(f_v), v \in \CG$ have been reached on all the lamp graphs $\Hh_v, v\in G$-s. Thus, by definition of the halting states, all the strong stationary times $\tau_{\Hh}(v)$ have happened. Then, by its definition, $\tau^\diamond$ has happened as well.
\end{proof}
Recall the definition \eqref{eq::tauhg} of $\tau^diamond$. Then we can construct a strong stationary time for $H\wr G$, described in the next lemma.
\begin{lemma}\label{lem::opt_sst_lamplighter_2}
Let $\tau^\diamond$ be the stopping time defined as in Lemma \ref{lem::opt_sst_lamplighter}, and let $\tau_G(x)$ be a strong stationary time for $G$ starting from $x\in G$ and define $\tau^\diamond_2$ by
\be\label{eq::tauhg2} \tau^\diamond_2:= \tau^\diamond+\tau_G(X_{\tau^\diamond}), \ee
where the chain is re-started at $\tau_G$ is started from $(\un F_{\tau_diamond},X_{\tau^diamond})$, run independently of the past and $\tau_G$ is measured in this walk.
Then, $\tau^\diamond_2$ is a strong stationary time for $H \wr G$.
\end{lemma}
\begin{proof}[Proof of Lemma \ref{lem::opt_sst_lamplighter_2}]
The intuitive idea of the proof is based on the fact that $\tau_G$ is conditionally independent of $\tau_H$-s and thus the lamp graphs stay stationary after reaching $\tau^\diamond$, and stationarity on $G$ is reached by adding the term $\tau_{G}(X_{\tau^\diamond})$. The proof is not very difficult but it needs a delicate sequence of conditioning. To have shorter formulas, we write shortly $\Pv$ for $\Pv_{(\un f_0,x_0)}$. First we condition on the events $\{ \tau^\diamond=s, X^\diamond_s=(\un g, y) \}$ and make use of \eqref{eq::tau_1_eq} from Lemma \ref{lem::opt_sst_lamplighter}.
\be\label{eq::tau_2} \ba \Pv\left[ X^\diamond_t = (\un f, x) , \tau_2^\diamond = t \right]= \sum_{s\le t; (\un g,y)} & \Pv\left[ X^\diamond_t = (\un f, x) , \tau_2^\diamond = t | \tau^\diamond=s, X^\diamond_s=(\un g, y) \right] \\ & \cdot \prod_{v \in G} \pi_H(g_v) \cdot \Pv\left[\tau^\diamond =s, X_s=y \right]. \ea \ee
Now for the conditional probability inside the sum on the right hand side we have
\[\ba &\Pv\left[ X^\diamond_t = (\un f, x) , \tau_2^\diamond = t | \tau^\diamond=s,X^\diamond_s=(\un g,y) \right] \\
& =\Pv\left[ X^\diamond_t = (\un f, x); \tau_G(y)\circ\theta_s= t-s| \tau^\diamond= s, X^\diamond_s=(\un g,y) \right]  \ea\]
where $\tau_G(y)\circ\theta_s$ means the time-shift of $\tau_G(y)$ by $s$, and we also used that $\tau_G$ is only depending on $y$.
We claim that
\[ \ba \sum_{\un g}&\left( \Pv_{(\un g, y)}\left[X^\diamond_{t-s}=(\un f, x), \tau_G (y)=t-s\right] \prod_{v \in G} \pi_H(g_v)\right)\\&
 =  \Pv_y[X_{t-s}=x,\tau_G(y)=t-s]\prod_{v \in G} \pi_H(f_v)\\
 &= \pi_G(x) \Pv_y[\tau_G =t-s]\prod_{v \in G} \pi_H(f_v).\ea \]
The first equality holds true since $\tau_G(y)$ is independent of the lampgraphs and the transition rules of $X^\diamond$ on $H\wr G$ tells us that the lamp-chains stay stationary. We omit the details of the proof. The second equality is just the strong stationarity property of $\tau_G$.
Thus, using this and rearranging the order of terms on the right hand side of \eqref{eq::tau_2} we end up with
\[ \sum_{s\le t, y\in G}  \Pv_y[\tau_G=t-s] \Pv[\tau^\diamond=s, X_s=y] \cdot \pi_G(x)\prod_{v \in G} \pi_H(f_v). \]
Then, realizing that the sum is just $\Pv[\tau^\diamond +\tau_G(X_{\tau^\diamond}) =t]$ finishes the proof.
\end{proof}
We continue with a lemma which relates the separation distance to the total variation distance:
Let us define first
\be\label{def::dxt}  d_x(t):= \| P^t(x,\cdot)-\pi(\cdot)\|_{\text{TV}} = \frac12 \sum_{y\in \Omega} \left|P^t(x,y)-\pi(y)\right|. \ee
The total variation distance of the chain from stationarity is defined as:
\[ d(t):= \max_{x\in \Omega} d_x(t). \]
The next lemma relates the total and the separation distance:
\begin{lemma}\label{lem::sep_tv}
For any reversible Markov chain and any state $x \in \Omega$, the separation distance from initial vertex $x$ satisfies:
\begin{align}
d_x(t) &\le s_x(t) \label{eq::tvlesep}\\
s_x(2t) & \le 4 d(t) \label{eq::sepletv}
\end{align}
\end{lemma}
\begin{proof}
For a short proof of \eqref{eq::tvlesep} see \cite{AF02} or \cite[Lemma 6.13]{LPW08}, and combine \cite[Lemma 19.3]{LPW08} with a triangle inequality to conclude \eqref{eq::sepletv}.
\end{proof}

We will also make use of the following lemma: (\cite[Corollary 12.6]{LPW08})
\begin{lemma}\label{lem::dtlimlambda}
For a reversible, irreducible and aperiodic Markov chain,
\[ \lim_{t\to \infty} d(t)^{1/t}= \la_*,\]
with $\la_*= \max\{|\la|: \la \text{ eigenvalue of P, } \la \neq 1\}.$
\end{lemma}
The two fundamental steps to prove Lemma \ref{lem::dtlimlambda} are the inequalities stating that for all $x\in \Omega$ we have
\be\ba\label{eq::la_lim_upper} d_x(t) &\le s_x(t) \le  \frac{\la_*^t}{\pi_{\min}},\\
|\la_2|^t&\le 2 d(t)\ea \ee
with $\pi_{\min}= \min_{y\in \Omega} \pi(y).$  This inequality follows from \cite[Equation (12.11), (12.13)]{LPW08}.

We note that Lemma \ref{lem::sep_tv} implies that the assertion of Lemma \ref{lem::dtlimlambda} stays valid if we replace $d(t)^{1/t}$ by the separation distance $s(t)^{1/t}$.
\section{Relaxation time bounds}\label{sec::trel}

\subsection{Proof of the lower bound of Theorem \ref{thm::main_relax}}\label{sec::trel_lower}

We prove $c_1=1/(16\log 2)$ in the lower bound of the statement of Theorem \ref{thm::main_relax}.
First note that it is enough to prove that  $\thit (\CG)$ and $|\CG|\trel (\Hh)$ are both lower bounds, hence their average is a lower bound as well. First we start showing the latter.

Let us denote the second largest eigenvalue of $Q$ by $\la_{\Hh}$ and the corresponding eigenfunction by $\psi$. It is clear that $\Ev_{\pi_H}(\psi)=0$ and we can normalize it such that $\Vv_{\pi_H}(\psi)=\Ev_{\pi_H}(\psi^2)=1$ holds.
Let us define
\[ \phi: V(\Hh \wr \CG)\to \Rb, \quad \phi((\un f,x))=\sum_{w\in \CG}\psi(f_w),\]
thus $\phi$ is actually not depending on the position of the walker, only on the configuration of the lamps.
Let $X^\diamond_t=(\un F_t, X_t)$ be the lamplighter chain with stationary initial distribution $\pid$. In the sequel we will calculate the Dirichlet form \eqref{def::dirichlet} for $\phi$ at time $t$, first conditioning on the path $X[0,t]$ of the walker:
\be\label{eq::dirichlet1} \ba \mathcal E_t [\phi] &= \frac12 \Ev_{\pid}[(\phi(X^\diamond_t)-\phi(X^\diamond_0))^2]\\
&=\frac12 \Ev_{\pid}\big(\Ev_{\pid}[(\phi(X^\diamond_t)-\phi(X^\diamond_0))^2| X[0,t]]\big) \ea \ee
We remind the reader that in each step of the lamplighter walk, the state of the lamp graph $\Hh_v$ is refreshed both at the departure and arrival site of the walker. Thus, knowing the trajectory of the walker implies that we also know $L_v(t)$, the number of steps made by the Markov chain $Z_v$ on $\Hh_v$.  Moreover, the collection of random walks $\left(Z_v\right)_{v\in \CG}$ on the lamp graphs are independent given $L_v(t)$-s.

We can calculate the conditional expectation on the right hand side of \eqref{eq::dirichlet1} by using the argument above and the fact that $\E_{\pi_H}(\psi)=0$ as follows:
\begin{align} \Ev_{\pid}\left[(\phi(X^\diamond_t)-\phi(X^\diamond_0))^2| X[0,t]\right]&=\!\!\sum_{v\in\CG}\!\! \Ev_{\pid}\!\!\left[\big(\psi(Z_v(L_v(t)) - \psi(Z_v(0))\big)^2 \big|L_v(t)\right] \label{eq::conditional_dirichlet}
\end{align}
Next, the product form of the stationary measure $\pid$ ensures that we can move to $\pi_H$ inside the sum and condition on the starting state $Z_v(0)$:
 \[ \ba \Ev_{\pid}&\!\!\left[\left(\psi\left(Z_v(L_v(t))\right) - \psi\left(Z_v(0)\right)\right)^2 \big|L_v(t)\right]\\
 = &2 \Ev_{\pi_H} \psi^2 -2 \Ev_{\pi_H}\left[\psi\left(Z_v(0)\right)\Ev_{Z_v(0)}[\psi\left(Z_v(L_v(t))\right) | Z_v(0)]\right], \ea \]
 Since $\psi$ was chosen to be the second eigenfunction for $Q$, clearly \\$\Ev_{Z_v(0)}[\psi\left(Z_v(L_v(t))\right)] = \la_H^{L_v(t)} \psi(Z_v(0))$. Using the normalization \\$\Ev_{\pi_H}[\psi^2]=1$, we arrive at
\begin{align}
 \Ev_{\pid}\left[\left(\phi(X^\diamond_t)-\phi(X^\diamond_0)\right)^2| X[0,t]\right]&= 2|G|-2 \sum_{v\in \CG} \la_{\Hh}^{L_v(t)}\label{eq::cond_dirichlet}
\end{align}
Since $\sum_{v\in \CG} L_{v}(t) = 2t$ and the function $\la_{\Hh}^x$ is convex, Jensen's inequality implies that
\[ \sum_{v\in G}\la_{\Hh}^{L_v(t)}\ge|G|\cdot\la_{\Hh}^{2t/|G|} .\]
Combining this with \eqref{eq::cond_dirichlet} and \eqref{eq::dirichlet1} and setting $t:=t^*= |\CG |\trel(H)= |G|/(1-\la_H)$ we arrive at
\[ \mathcal E_t(\phi) \le |G|\left(1-e^{2  \frac{\log\la_H}{1-\la_H}}\right)\le |G|\left(1- 2^{-4 -1}\right),\]
where in the last step we assumed $\la_{\Hh}>1/2$, since in this case we have $(1-\la_{\Hh})^{-1}\log\la_{\Hh}> -2 \log 2$.  On the other hand, if $\la_{\Hh}<1/2$, than $\trel(H)<2$ and we will use the other lower bound $t_{\text{hit}}(\CG)$ which is at least of order $|G|$.
Dividing by $\Vv_{\pid} \phi=|G|$, and using the variational characterization of the spectral gap in Lemma \ref{lem::variational}, we get that the spectral gap $\gamma_{t^*}$ at time $t^*$ satisfies
\[ \gamma_{t^*}\le 1-2^{-5}. \]
Since $\gamma_t$ is by definition the spectral gap of the chain at time $t$, we have
\be\label{eq::spectralgap} 1-\la_2(\Hh\wr\CG)^{t^*} \le 1-2^{-5}.\ee
Thus
\[ 5 \log2 \ge t^* \left(1-\la_2(\Hh\wr\CG)\right),\]
so we get a lower bound $\trel(\Hh\wr\CG)\ge \frac{1}{5 \log 2} |\CG| \trel(\Hh)$.

To get the lower bound $t_{\text{hit}}(G)/4$ we adjust the proof for $0-1$ lamps ($\Hh=\Z_2$) \cite[Theorem 19.1]{LPW08} to our setting.
First pick a vertex $w\in \CG$ which maximizes the expected hitting time $\Ev_{\pi_G}(\tau_w)$.  As before, we will use the second eigenfunction $\psi$ with eigenvalue $\la_H$ with $\Ev_{\pi_H}(\psi)=0, \Ev_{\pi_H}(\psi^2)=1$ and define
\[ \phi\left((\un f,x)\right):=\psi(f_w). \]
Easy to see with the same conditioning argument we used in \eqref{eq::conditional_dirichlet} and \eqref{eq::cond_dirichlet} that the Dirichlet form at time $t$ equals
\[ \mathcal E_t(\phi)= 1-\Ev_{\pid}\left[\la_H^{L_w(t)}\right]\]
Now we will show that $\Ev_{\pid}\left[\la_H^{L_w(t)}\right]\ge 1/4$.
To see this we first note that for any $t$ we have for the hitting time $\tau_w$ of $w\in G$
\[ \ba \E_v(\tau_w) &\le t + t_{\text{hit}} \Pv_v[\tau_w>t] \\
\E_{\pi_G}(\tau_w) &\le t + t_{\text{hit}} \Pv_{\pi_G}[\tau_w>t] \ea \]
  To see the first line: either the walk hits $w$ before time $t$, or the expected additional time it takes to arrive at $w$ is bounded by $\thit$ regardless of where it is at time $t$. The second line follows by averaging over $\pi_G$.

Next, \cite[Lemma 10.2]{LPW08} states that $\thit \le 2 \max_v \Ev_{\pi}[\tau_v]$ holds for every irreducible Markov chain. We exactly picked $w$ such that it maximizes $\Ev_{\pi_G}(\tau_v)$, so we have $t_{\text{hit}}\le 2 \E_ {\pi_G}[\tau_w]$, so multiplying the previous displayed inequality by 2 gives
\[ t_{\text{hit}} \le 2t + 2 t_{\text{hit}} \Pv_{\pi_G}[\tau_w>t] \]
Now substituting $t=\thit/4$ and rearranging terms results in
\[  \Pv_{\pi_G}\left[\tau_w>\frac{\thit}{4}\right] \ge \frac14. \]
Since $\{ L_w(\thit/4)=0 \}= \{ \tau_w>\thit/4 \}$, we can use this inequality to obtain the upper bound
\[ \mathcal E_{\thit/4}(\phi)= 1- \E_{\pid}\left[\la_{\Hh}^{L_w(\thit/4)}\right] \le 1-\Pv_{\pi_G}[\tau_w>\thit/4 ] \le 1- \frac14 = \frac34. \]
Analogous to the last lines of the proof of the lower bound above, (see \eqref{eq::spectralgap}) we obtain the other desired lower bound:
\[ \trel(\Hh) \ge  \frac{1}{2\log 2} \frac14 \thit(\CG) .\]
Putting together the two bounds we get
\[ \ba\trel(\Hh\wr\CG) &\ge \max\left\{\frac{1}{8\log 2} \thit(\CG), \frac{1}{5\log2}|\CG|\trel(\Hh)\right\} \\
&\ge \frac{1}{16\log 2} \left(\thit(\CG)+ |\CG|\trel(\Hh)\right).\ea \]

\subsection{Proof of the upper bound of Theorem \eqref{thm::main_relax}}\label{sec::trel_upper}
To prove the upper bound, we will estimate the tail behavior of the strong stationary time $\tau^\diamond_2$ in Lemma \ref{lem::opt_sst_lamplighter_2}, relate it to $s^\diamond(t)$, the separation distance on $H\wr G$, and then use Lemmas \ref{lem::dtlimlambda} and \ref{lem::sep_tv} to see that $s^\diamond(t)^{1/t} \to \lambda^\diamond$.  We will use separation-optimal $\tau_H$ and $\tau_G$ in the construction of $\tau_2^\diamond$.  The existence is guaranteed by Theorem \ref{thm::AD}. We will use $\Pv$ for $\Pv_{(\un f, x)}$ for notational convenience.
Combining \eqref{eq::sep_ineq} and the fact that $\tau^\diamond$ happens when all the stopping times $\tau_{\Hh}(v), v\in \CG$ have happened on the lamp graphs, by union bound we have for any choice of $0<\al<1$
\begin{align} s_{(\un f, x)}^{\diamond}(t)&\le \Pv_{(\un f,x)}\left[\tau^\diamond_2>t\right] \le \Pv_{(\un f,x)}\left[\tau^\diamond>\al t\right]+ \Pv_{(\un f,x)}\left[\tau_2^\diamond>t| \tau^\diamond < \al t\right]\notag \\
&\le \ \Pv[\tau_{\text{cov}}>\al t/3]\label{eq::cov_not}\\
&\ + \Pv\left[ \exists w\in \CG:   L_w(\al t)< \tfrac{\al t}{2|\CG|} \big| \tau_{\text{cov}} \le \al t/3\right]\label{eq::localtime_not}\\
&\ + \Pv\left[\exists w\in \CG: \tau_{\Hh}(w)> L_w(\al t)\ \big|  \forall v\in \CG: L_v(\al t)\ge\frac{\al t}{2|\CG|} \right] \label{eq::tauh_not}\\
&\ + \max_{(\un g, y)}\Pv_{(\un g,y)}\left[\tau_G>(1-\al)t\right] \label{eq::fourth_term}
\end{align}
Namely, there are four possibilities:  The first option is that there is a state $w\in \CG$ which is not hit yet, i.e. the cover time of the chain is greater than $\al t/3$: giving the term \eqref{eq::cov_not}. The constant $1/3$ could have been chosen differently, we picked $\al t/3$ such that the remaining $2\al t/3$ time still should be enough to gain large enough local time on the vertices $v\in G$. Secondly, even though any state $w$ on the graph $\CG$ is reached before time $\al t/3$, the remaining time was not enough to have at least $\al t/2|G|$ many moves on some lamp graph $H(w)$, term \eqref{eq::localtime_not}. The third option is that even though there have been many moves on all the lamp graphs, there is a vertex $w\in G$ where $\tau_{\Hh}(w)$ has not happened yet, yielding the term \eqref{eq::tauh_not}.
We will handle the three terms separately. The fourth term handles the case where the strong stationary time $\tau_G$ is too large. (For convenience, we will write $t$ instead
of $\alpha t$ in estimating the first three formulas.)

We can estimate the first term \eqref{eq::cov_not} by a union bound:
\be\label{eq::not_cov_estimate} \Pv[\tau_{\text{cov}}>t/3] \le \Pv[\exists w: \tau_w>t/3] \le |\CG| 2 e^{-\tfrac{\log 2}{6}\tfrac{t}{\thit} },\ee
where $\thit$ is the maximal hitting time of the graph $\CG$, see \eqref{eqn::thit_definition}.
To see this, use Markov's inequality on the hitting time of $w\in \CG$ to obtain that for all starting states $v\in \CG$ we have
$\Pv_v[ \tau_w> 2\thit ] \le 1/2 $,  and then run the chain in blocks of $2\thit$. In each block we hit $w$ with probability at least $1/2$, so we have
\[ \Pv_v[ \tau_w> K (2\thit) ] \le \frac{1}{2^K}.\]
To get it for general $t$,  we can move from $\fl{ t/\thit}$ to $t/\thit$ by adding an extra factor of $2$, and \eqref{eq::not_cov_estimate} immediately follows by a union bound.

For the third term \eqref{eq::tauh_not} we claim the following upper bound holds:
\be\label{eq::not_tauh_estimate}  \Pv\left[\exists w: \tau_{\Hh}(w)\ge L_w(t)\big| \forall v:L_v(t)>\tfrac{t}{2|G|} \right] \le |G| \frac{1}{\pi_{\min}(\Hh)} e^{-\tfrac{t}{2|G| \trel(H)}}.\ee
To see this we estimate the probability of the event $\{\tau_{\Hh}(w)\ge L_w(t)\ \big|  L_w(t)\ge\frac{t}{2|\CG|} \}$ on a single lamp graph and then use a union bound to lose a factor $|\CG|$ and arrive at the right hand side.
First note that according to Lemma \ref{lem::dtlimlambda}, the tail of the strong stationary time $\tau_{\Hh}$ is driven by $\la_{\Hh}^t$. More precisely, using the inequality \eqref{eq::la_lim_upper} we have that for any initial state $h\in \Hh$:
\[ \ba \Pv_h\left[ \tau_H(w) \ge \frac{t}{2|\CG|} \right] &\le s_{\Hh}\left(\tfrac{t}{2|\CG|}\right) \le  \frac{1}{\pi_{\min}(\Hh)} \la_{\Hh}^{t/2|\CG|} \\
&\le \frac{1}{\pi_{\min}(\Hh)} \exp\left\{-\frac{(1-\la_{\Hh}) t}{2|\CG|}\right\}.  \ea \]
Since we have made at least $L_w(t)\ge \tfrac{t}{2|\CG|}$ steps on each coordinate, the claim \eqref{eq::not_tauh_estimate} follows. The fourth term \eqref{eq::fourth_term} can be handled analogously and yields an error probability $\exp\{ - c t/\trel(G) \} $ which then, taking the power of $1/t$ and limit as in Lemma \ref{lem::dtlimlambda}, will lead to a term of order $\trel(G)$. Then, taking into account that  $\trel(G)\le c \tmix(G)\le C \thit(G)$ holds for any lazy reversible chain (see e.g. \cite[Chapter 11.6,12.4]{LPW08}), we can ignore this term.

  The intuition behind the estimates below for the second term  \eqref{eq::localtime_not} is that since the total time was at least $2t/3$ after hitting,  regularity of $\CG$ implies that the average number of moves on a lamp graph equals $4t/(3|G|)$ by the double refreshment at any visit to the vertex.  Thus, the probability of having less than $t/(2|G|)$ moves must be small.

More precisely, we introduce the excursion-lengths to a vertex $w\in \CG$:
Let us define for all $w\in \CG$ the first return time to state $w$ as
\[ R(w)=\inf\{ t>0: X_t=w|X_0=w\}.\]
The strong Markov property implies that the length of the $i$-th excursion $R_i(w)$, defined as the time spent between the $(i-1)$th and $i$th visit to $w$, are i.i.d random variables distributed as the first return time $R(w)$.

Thus, having not enough local time on some site $w\in \CG$ can be expressed in terms of the excursion lengths $R_i(w)$-s as follows:
\be\label{eq::move_to_excursions}  \Pv\left[ \exists w:   L_w(t)\le \tfrac{t}{2|\CG|} \big| \tau_{\text{cov}} \le \frac{t}{3} \right]\le |\CG| \max_{w\in G}\Pv_w\left[ \sum_{i=1}^{t/4|\CG|} R_i(w)\ge \frac{2t}{3} \right] , \ee
since conditioning on hitting before $t/3$ ensures that we had at least $2t/3$ steps to gain the $t/4|G|$ visits to $w$, and by the definition \eqref{def::localtime} of $L_v(t)$, this guarantees that $L_w(t)<t/2|G|$.

  We aim to estimate the right hand side of \eqref{eq::move_to_excursions} using the moment generating function of the first return time $R(w)$. To be able to carry out the estimates we need a bound on the tail behavior of the return times. A very similar argument can be used to the one we used for the tail of the cover time \eqref{eq::not_cov_estimate}, namely the following holds:
\[ \ba \Pv_w\left[R(w)>2\thit+1\right] &= \Pv_w\left[X_1\neq w\right]\Ev\left[ \Pv_{X_1} (\tau_w>2\thit|X_1) \right]\\
&\le \frac{\Ev[\Ev_{X_1}(\tau_w) ] }{  2\thit}\le \frac12. \ea  \]
Running the chains in blocks of $2\thit+1$, one can see that in each block the chain has a chance at least $1/2$ to return to $w$, so we have for each $t>2\thit+1$
\be \label{eq::excursion_decay_rate} \Pv\left[ R_w> t \right] \le 2 \left(\frac{1}{2}\right)^{ \tfrac{t}{2\thit+1}}= 3 \exp\left\{ -\frac{\log2}{2} \frac{t}{\thit}\right\}, \ee
where the factor $3$ comes from ignoring to take the integer part of $t/\thit$ and neglecting the $+1$ term in the denominator.

We can use this tail behavior to estimate the moment generating function
\[\ba \Ev\left[ e^{\beta R_w} \right] &\le e^{2 \beta |\CG|} + \Ev[ e^{\beta R_w}\ind\{R_w>2|\CG|\}] \\
&= e^{2\beta |\CG|} + \int_{e^{2\beta|\CG|}}^\infty \Pv[e^{\beta R_w}>z]\mathrm dz \ea \]
where we cut the expectation at $2|G|$. Using the bounds in \eqref{eq::excursion_decay_rate} yields:
\[ \ba
\Ev\left[ e^{\beta R_w} \right]&\le e^{2\beta | \CG|}+ \int_{e^{2\beta|\CG|}}^\infty \Pv[ R_w>\frac{1}{\beta}\log z]  \mathrm dz\\
&\le e^{2\beta |\CG|} +  2 \int_{e^{2\beta|\CG|}}^\infty z^{-\tfrac{\log 2}{2 \beta \thit}}  \mathrm dz.\ea \]
 Setting arbitrary $\beta< \log 2/ (2\thit)$ makes the second term integrable, and with the special choice of $\beta=\frac{\log2}{4\thit}$ we obtain the following estimate:
\be\label{eq::exp_moment_recursion_est}
\Ev\left[ e^{\beta R_w} \right] \le e^{ 2 \beta |\CG|}  + 2 e^{-2 \beta |\CG|} \le e^{(2+\delta) \beta |\CG|}
\ee
with an appropriately chosen $0<\delta<1/3$.
Now we apply Markov's inequality to the function $e^{\beta \sum_{i=1}^{t/4|G|} R_i(w)}$  to estimate the right hand side of \eqref{eq::move_to_excursions}:
 \be  \Pv_w\left[ \sum_{i=1}^{t/4|G|} R_i(w)\ge 2t/3 \right] \le e^{-\tfrac23 \beta t}\cdot \Ev\left[ e^{\beta R_w}\right] ^{\tfrac{t}{4|G|}},\ee
where we also used the independence of the excursions $R_i(w)$-s.  Using the estimate in \eqref{eq::exp_moment_recursion_est} to bound the right hand side we gain that
\be \ba \Pv_w\left[ \sum_{i=1}^{t/4|\CG|} R_i(w)\ge 2t/3 \right] &\le  e^{-\tfrac23 \beta t}\cdot e^{(2+\delta) \beta |G| \cdot \tfrac{t}{4| \CG|} } \\
&\le e^{-\tfrac{1-\tilde \delta }{6}\beta t}=  \exp\left \{-\frac{(1-\tilde \delta)\log2 }{24} \frac{t}{\thit} \right\},
\ea\ee
where we used $\beta=\log2/(4\thit)$, and modified $\tilde\delta:= 3\delta/2\le 1/2.$
Using the relation of the local time to the excursion lengths in \eqref{eq::move_to_excursions} we finally get that the second term \eqref{eq::localtime_not} is bounded from above by
\be\label{eq::localtime_estimate2} \Pv\left[ \exists w:   L_w(t)\le \frac{t}{2|\CG|} \big| \tcov \le t/3\right] \le |\CG| \exp\left\{-\frac{\log2 }{48} \frac{t}{\thit}  \right\}. \ee
Mind that all the estimates \eqref{eq::not_cov_estimate}, \eqref{eq::not_tauh_estimate} and  \eqref{eq::localtime_estimate2} were independent of the initial state $(\un f, x) \in H\wr G$, so using the second inequality in \eqref{eq::la_lim_upper} and maximizing over all possible initial states yields us
\be \ba |\la_2|^t\le 2d^\diamond(t) &\le 2 s^\diamond(t) \le  \frac{4|\CG|}{\pi_{\min}(H)} \exp\left\{-\frac{t}{2|\CG| \trel(\Hh)}\right\} \\
&+4  \exp\left \{-\frac{(1-\tilde \delta)\log2 }{24} \frac{t}{\thit} \right\} + 4 |\CG| \exp\left\{-\frac{\log2}{6}\frac{t}{\thit} \right\}
\ea\ee\label{eq::la_estimate}
In the final step we apply Lemma \ref{lem::dtlimlambda}: we  take the power $1/t$  and limit as $t$ tends to infinity with fixed graph sizes $|\CG|$ and $|\Hh|$ on the right hand side of \eqref{eq::la_estimate} to get an upper bound on $\la_2$.  Then we use that $(1-e^{-x})\le x +o(x)$ for small $x$ and obtain the bound on $\trel$ finally:
\[ \trel(H\wr G) \le \max \left \{ 2 |\CG| \trel(\Hh), \frac{48}{\log 2 } \thit \right\}.\]
This finishes the proof of the upper bound on the relaxation time.

\section{Mixing time bounds}\label{sec::tmix}
  Based on the fact that $H$ has a separation-optimal strong stationary time $\tau_H$, the idea of the proofs is to relate the separation distance to the tail behavior of the stopping times $\tau^\diamond$ and $\tau_2^\diamond$ constructed in Lemmas \ref{lem::opt_sst_lamplighter} and \ref{lem::opt_sst_lamplighter_2}, respectively. Then these estimates are turned into bounds of the total variation distance using the relations in Lemma \ref{lem::sep_tv}. This method gives us the upper bound in \eqref{eqn::tmix_main} and the corresponding lower bound under the assumption {\bf (A)}. For the lower bound without the assumption, we will need slightly different methods.
\subsection{Proof of the upper bound of Theorem \ref{thm::main_mixing}}
The idea of the proof is to use appropriate top quantiles of the strong stationary time $\tau_{\Hh}$ on $\Hh$, and give an upper bound on the tail of the strong stationary time $\tau^\diamond_2$ defined in Lemma \ref{lem::opt_sst_lamplighter_2}. Throughout, we (only) need that $\tau_H$ and $\tau_G$ in the construction of $\tau_2^\diamond$ are separation-optimal. The existence is guaranteed by Theorem \ref{thm::AD}. (Thus, $\tau_H$ does not necessarily possess halting states.)

Let us denote the worst-case initial state top $\ve$-quantile of a stopping time $\tau$ as \be\label{eq::quantile_def}\tqu{\ve}(\tau):= \max_{y\in \Omega} \inf \{t: \Pv_y[\tau >t] \le \ve \}\ee
We continue with the definition of the blanket time:
\be\label{eq::blanket_def} \mathcal B_2:= \inf_t\left\{\forall v,w\in G: \frac{L_w(t)}{L_v(t)}\le 2 \right\}.\ee
Let us further denote
 \be\label{eq::expected_blanket_def} \B_2:= \max_{v\in G} \Ev_v(\mathcal B_2) \ee
 It is known from \cite{DLP11} that there exist universal constants $C$ and $C'$ such that $C'\tcov \le  \B_2 \le C \tcov$.

Thus, our first goal is to show that at time \[ 8 \B_2 + |\CG| \tqu{1/16|\CG|}(\tau_{\Hh}) + \tqu{1/16}(\tau_{\CG})=:8 \B_2 + |G|t_H^u + t_G =: \td \] we have for any starting state $(\un f, x)$ that
\be\label{eq::tautail} \Pv_{(\un f,x)}[\tau^\diamond_2>\td]\le \frac14. \ee
We remind the reader that $\tau_2^\diamond=\tau^\diamond+\tau_G(X_{\tau^\diamond})$ and thus the following union bound holds:
\be\label{eq::upper_cond2}\ba  \Pv\left[\tau^\diamond_2>\td\right]& \le \Pv\left[\mathcal B_2> 8\B_2 \right] + \Pv\left[ \tau^\diamond > |G|t_H^u+ 8 \B_2 | \mathcal B_2\le 8 \B_2  \right] \\
& \quad +  \max_{v\in G}\Pv_v\left[ \tau_G > t_G | \mathcal B_2 \le 8 \B_2, \tau^\diamond< 8 \B_2 + |G| t_H^u \right],\ea \ee
where in the third term we mean that we restart the chain after time $8\B_2+ |G| t_H^u$, and measure $\tau_G$ starting from there. The first term on the right hand side is less than $1/8$ by Markov's inequality, the third is less than $1/16$ by the definition of the worst case quantile. The second term can be handled by conditioning on the local time sequence of vertices and on the blanket time: (for shorter notation we introduce $t_1:=|G|t_H^u+ 8\B_2$)
\be\label{eq::upper_cond} \ba &\Pv\left[ \tau^\diamond >|G|t_H^u+ 8\B_2| \mathcal B_2\le 8 \B_2 \right] \\
 &=\!\!\!\!\!\! \sum_{s\le 8\B_2, \left(L_v(t_1)\right)_v}\!\!\!\!\!\!\!\! \Pv\left[\exists w\!: \{\tau_{\Hh}(w) > L_w(t_1)\} \big| \left(L_v(t_1)\right)_v, \mathcal B_2=s\right]\!\cdot\! \Pv\left[ \left(L_v(t_1)\right)_v, \mathcal B_2=s\right] \ea\ee
 The fact that $\mathcal B_2 \le 8 \B_2$ means that the number of visits to every vertex $v\in G$ must be greater than half of the average, which is at least $\tfrac{1}{2} t_{\Hh}^u$. Since $L_w(t)$ is twice the number of visits by \eqref{def::localtime},   $\{\tau_{\Hh}(w) >L_w(t_1)\} \subseteq \{\tau_{\Hh}(w) >t_{\Hh}^u\}$. By the definition of the quantiles,
 \[ \Pv_h\left[\tau_{\Hh}(w)>t_{\Hh}^u\right]\le \frac{1}{16 |\CG|} \]
holds for every $h\in \Hh$ and $w\in \CG$.  Applying a simple union bound on the conditional probability on the right hand side of \eqref{eq::upper_cond} yields
\[ \ba \Pv_{(\un f, x)}\left[\tau^\diamond>t_1 | \mathcal B_2\le 8 \B_2  \right] &\le
 \sum_{s\le 8\B, \left(L_v(t_1)\right)_v}\!\!\!\!\! \left(|\CG| \frac{1}{16|\CG|} \right) \Pv\left[ \left(L_v(t_1)\right)_v, \mathcal B_2=s\right]\\
&\le \frac{1}{16}, \ea \] where we used that the sum of the probabilities on the right hand side is at most 1.
Combining these estimates with \eqref{eq::upper_cond2} yields \eqref{eq::tautail}.
It remains to relate the worst-case quantiles to the total variation mixing times. Here we will make use of the separation-optimal property of $\tau_{\Hh}$ and $\tau_{\CG}$.
Now just consider the walk on $G$. Let us start the walker on $\CG$ from an initial state $x_0 \in \CG$ for which the maximum is attained in the definition \eqref{eq::quantile_def} of the quantile $\tqu{1/16}(\tau_{\CG})$. Then, by \eqref{eq::sepletv} we have that one step before the quantile we have
\[  \ba \frac{1}{16} \le  \Pv_{x_0}\left[\tau_{\CG}> \tqu{1/16}(\tau_{\CG})-1\right] &= s_{x_0}\left(\tqu{1/16}(\tau_{\CG})-1\right)\\
&\le 4 d\left(\frac12 (\tqu{1/16}(\tau_{\CG})-1
)\right). \ea\]
This immediately implies that $ \tfrac{1}{2}(\tqu{1/16}(\tau_{\CG})-1)\le \tmix\left(\CG, \tfrac{1}{64}\right).$ By the submultiplicative property of the total variation distance $d(k t)\le 2^k d(t)^k$ we have that $\tmix(\CG, \tfrac{1}{64})\le 6 \tmix\left(\CG, \tfrac{1}{4}\right) $. So we arrive at
\be \label{eq::tvmix_upper_G} \tqu{1/16}(\tau_{\CG})-1 \le 12 \tmix\left(\CG\right) \ee
Similarly, starting all the lamps from the position $h_0$ where the maximum is attained in the definition of $t_H^u=\tqu{1/16|\CG|}(\tau_{\Hh})$, one step before the quantile we have
\[  \frac{1}{16|\CG|}\le \Pv_{h_0}\left[\tau_{\Hh}> t_{\Hh}^u-1 \right] = s_{h_0}\left(t_{\Hh}^u-1\right)\le 4 d\left((t_{\Hh}^u-1)/2\right) \]
So we have
\be \label{eq::tvmix_upper_H}  \frac12 (\tqu{1/16|\CG|}(\tau_{\Hh})-1)\le \tmix\left(\Hh, \tfrac{1}{64|\CG|}\right).\ee

On the other hand, on the whole lamplighter chain $\Hh \wr \CG$ we need the other direction: For every starting state $(\un f, x)$ \eqref{eq::tvlesep} and \eqref{eq::tautail} implies that
\[ d_{(\un f, x)}(\td)\le s_{(\un f, x)}(\td) \le \Pv_{(\un f,x)}\left[\tau^\diamond_2 > \td\right] \le 1/4 \]
Maximizing over all states $(\un f, x)$ yields
\be \label{eq::tvmix_upper_whole} \tmix( \Hh \wr \CG ) \le \td. \ee
Putting the estimates in \eqref{eq::tvmix_upper_G} and \eqref{eq::tvmix_upper_H} to \eqref{eq::tvmix_upper_whole},  we get that
 \[ \tmix(\Hh\wr \CG) \le \td \le 8 \B_2(\CG) + 12 \tmix(\CG)+1 + 2|\CG| \left(\tmix\left(H, \frac{1}{64 |G|}\right )+\frac12\right). \]
Since $\B_2(\CG)\le C \tcov(\CG)$, and $\tmix(\CG)\le 2 \thit(\G) \le 2 \tcov(\CG)$ for any $\CG$ (see for instance \cite{LPW08}), the assertion of Theorem \ref{thm::main_mixing} follows with $C_2=8 (C+3)$, where $C$ is the universal constant relating the blanket time $\B_2$ to the cover time  $\tcov$ in \cite{DLP11}.

We remark why we did not make the constant $C_2$ explicit:  If the blanket time $\B_2$ were not used in our estimates, the error probability that some vertex $w\in \CG$ does not have enough local time would need to be added. This, similarly to the term \eqref{eq::localtime_not} behaves like $|\CG| e^{- c \left(\tcov + |\CG|\tmix(\Hh, \frac{1}{\CG}) \right)/\thit }$. If we do not assume anything about the relation of $\thit(G)$ and $\tcov(G)$ and on $\tmix(H, \frac{1}{\CG})$, then this error term will not necessarily be small. For example, if $\CG_n$ is a cycle of length $n$, $\Hh_n$ is a sequence of expander graphs, then $\tcov(G_n) = \thit(G_n) = \Theta( n^2 ) $, and  $\tmix(H, \frac{1}{\CG})=\log |\Hh| \cdot \log |\CG|  = \log |\Hh| \log n $, and we see that the term is not small if $\log|\Hh| = o(n/\log n ).$

\subsection{Proof of the lower bounds of Theorem \ref{thm::main_mixing}}\label{sec::tmix_lower}

As we did with the relaxation time, it is enough to prove that all the bounds are lower bounds separately, then take an average. First we start showing that the upper bound is sharp in \ref{eqn::tmix_main} under the assumption that there is a strong stationary time $\tau_H$ with halting states.

\subsubsection{Lower bound under Assumption (A)}
We first aim to show that\\ $c\ |G| \tmix(H, \tfrac{1}{|G|}) \le \tmix(H\wr G)$.

Consider the stopping time $\tau^\diamond$ constructed in Lemma \ref{lem::opt_sst_lamplighter}.
Corollary \ref{cor::sep_lower_bound} tells us that the tail of $\tau^\diamond$ lower bounds the separation distance at time $t$. We again emphasize that this bound holds only if $\tau_H$ in the construction of $\tau^\diamond$  is not only separation optimal but it also \emph{has a halting state}. Our first goal is to lower bound the tail of $\tau^\diamond$, then relate it to the total variation distance.

First set
\be \label{eq::td_tH_def_lower}
t_{\Hh}^l:= \tqu{|G|^{-1/2}/2} (\tau_H)-1, \quad
\td:=\frac{1}{4} |G|  t_{\Hh}^\ell, \ee
clearly this time is nontrivial if $\tqu{|G|^{-1/2}/2} (\tau_H)\neq 1$. We handle the case if it equals $1$ later.
We can estimate the upper tail of $\tau^\diamond$ by conditioning on the number of moves on the lamp graphs $H_v, v\in \CG$:
\be \label{eq::uppertail}\ba \Pv\left[\tau^\diamond>\td \right] &\ge \Pv\left[ \exists w\in \CG: \tau_{\Hh}(w)>L_w(\td) \right]\\
&\ge \sum_{(L_v(\td))_v} \Pv\left[ \exists w\in \CG: \tau_{\Hh}(w)>L_w(\td) \big| (L_v(\td))_v\right] \Pv\left[(L_v(\td))_v\right]
\ea \ee
 For each sequence $(L_v(\td))_{v\in \CG}$ we define the random set
\[ S_{(L_v)_v}:= \left\{ w\in \CG: L_w(\td) \le t_{\Hh}^\ell\right\}\]
Since $\sum_v L_v(\td) = 2\td= \tfrac12|\CG|  t_{\Hh}^\ell$, we have that for arbitrary local time configuration $(L_v(\td))_v$,
\be\label{eq::size_small_localt} |S_{(L_v)_v}| \ge |\CG|/2. \ee
Thus we can lower bound \eqref{eq::uppertail} by restricting the event only to those $w\in G$ coordinates which belong to this set, i.e. whose local time is small:
\be \label{eq::upper2} \ba \Pv\left[\tau^\diamond>\td \right] &\ge \!\! \!\sum_{(L_v(\td))_v}\! \!\! \Pv\left[ \exists w\in S_{(L_v)_v}: \tau_H(w)>L_w(\td) \big| (L_v\left(\td)\right)_v\right] \Pv\left[\left(L_v(\td)\right)_v\right] \\
\ge & \!\! \sum_{\left(L_v(\td)\right)_v}\!\! \Pv\left[ \exists w\in S_{(L_v)_v}: \tau_H(w)>t_{\Hh}^\ell\big| \left(L_v(\td)\right)_v\right] \Pv\left[\left(L_v(\td)\right)_v\right],
 \ea\ee
 where in the second line we used that for $w\in S_{(L_v)_v}$ we have $\{\tau_H(w)>L_w(\td) \} \supseteq \{ \tau_H(w) > t_{\Hh}^\ell\}$. Conditioned on the sequence $(L_v(\td))_v$, the times $\tau_H(w)$ for $w\in S_{(L_v)_v}$ are independent. On each lamp graph $\Hh(v)$ let us pick the starting state to be $h_0\in \Hh$ where the maximum is attained in the definition of $\tqu{|\CG|^{-1/2}/2}(\tau_H)$. Since $t_H$ is one step \emph{before} the quantile, we have
 \be \label{eq::quantile_eq}\Pv_{h_0}\left[\tau_H(w) > \tqu{|\CG|^{-1/2}} (\tau_H)-1\right] \ge |\CG|^{-1/2}/2. \ee
 We need to start the lamp-chains from the worst-case scenario $h_0\in \Hh$ for two reasons: First, we needed to define the quantile as in \eqref{eq::quantile_def} to be able to relate it to the total variation mixing time on $\Hh$, see below. Then, the fact that $\tqu{\ve}$ was defined as the worst-case starting state quantile means that for other starting states the quantile may be smaller, and the lower bound can possibly fail.

 Combining \eqref{eq::quantile_eq} with \eqref{eq::size_small_localt} and the conditional independence gives us the following stochastic domination from below to the event in \eqref{eq::upper2}
\[ \Pv\left[ \exists w\in S_{(L_v)_v}: \tau_H(w)>t_{\Hh}^\ell \big| (L_w(\td))_w\right] \ge \Pv[V> 0],\]
where $V$ is a Binomial random variable with parameters $\left(|\CG|/2, |\CG|^{-1/2}/2
\right)$. Clearly, for $|\CG| > 8 > 16 (\log 2) ^2 $ we have
\[ \Pv\left[V>0\right]= 1-\left(1-\frac{1}{2|G|^{1/2}}\right)^{|G|/2} \ge 1- e^{-\tfrac{|G|^{1/2}}{4}}. \]
Combining this with  \eqref{eq::upper2} and summing over all possible $(L_v(\td))_{v\in G}$ sequences  we easily get that
\[  \Pv\left[\tau^\diamond>\td \right] \ge 1- e^{-\tfrac{|G|^{1/2}}{4}}.\]
Then, by Corollary \ref{cor::sep_lower_bound} we have
\[ s_{(\un h_0,x)}^\diamond(\td) \ge 1- e^{-\tfrac{|G|^{1/2}}{4}}.\]
In the next few steps we relate the tail of $\tau^\diamond$ and $\tau_H$ to the mixing time of the graphs. First, combining the previous inequality with \eqref{eq::sepletv} implies that for the starting state $(\un h_0,x)$ the following inequalities hold:
\[ 1- e^{-|\CG|^{1/2}/4}\le s_{(\un h_0,x)}^\diamond(\td)\le 4d^\diamond(t^\diamond/2).\]
These immediately imply
\be \label{eq::tmix_lower_td}\tmix(\Hh\wr \CG, \frac18 )\ge \frac12\td = \frac18 |\CG| t_{\Hh}^\ell\ee
Now we will relate $t_{\Hh}^\ell=\tqu{|\CG|^{-1/2}/2} (\tau_{\Hh})-1$ to the mixing time on $\Hh$.
Since $t_{H}$ investigates the worst case initial-state scenario, by inequality \eqref{eq::sep_ineq} for any starting state $h\in \Hh$ we have
\[ s_{h} (t_H+1)\le \Pv_{h}\left[\tau_{\Hh}\ge t_H+1\right] \le |\CG|^{-1/2}/2\]
Using $d_h(t)\le s_h(t)$ (see Lemma \ref{lem::sep_tv}) and maximizing over all $h\in \Hh$
we get that \be\label{eq::dt_H}d_H(t_H+1)\le |\CG|^{-1/2}/2.\ee
On the other hand, the total variation distance for any Markov chain has the following sub-multiplicative property for any integer $k$, see \cite[Section 4.5]{LPW08}:
\be\label{eq::tv_submultiplicative} d(k t)\le 2^k d(t)^k.\ee
 Taking $t= t_H+1$ and combining with \eqref{eq::dt_H} we have that
\[ d_H(2 (t_H+1)) \le 4 d_H(t_H+1)^2 \le 4 \frac{1}{4|\CG|}, \] which immediately implies
\[   \tmix(\Hh,  1/|\CG|)\le 2 (t_H+1) .\]
Combining this with \eqref{eq::tmix_lower_td} yields the desired lower bound:
\[  \frac{1}{16} |\CG| \left(\tmix\left(\Hh,  \tfrac{1}{|\CG|}\right)-2\right)\le \tmix(\Hh\wr \CG, \frac{1}{8}).  \]
Mind that the term $-2$ in the brackets can be dropped when picking a possibly smaller constant and take the graph large enough.
The case when $\tqu{|G|^{-1/2}/2} (\tau_H)=1$ can be handled the following way: first mind that we can exchange the quantile for arbitrary $0<\alpha<1$, and look at the proof with $\tqu{|G|^{-\alpha}/2} (\tau_H)$. If this is still $=1$ for all $\alpha$, that means that $\tau_H =1$ a.s. In this case, it is enough to hit the vertices to mix immediately and thus the mixing time $|\CG| \tmix(H)$ is of smaller order than the cover time $\tcov(\CG)$.
The case when $|\CG| \le 8$ but $|\Hh| \to \infty $ is easy to see since in this case $\tmix(H, \tfrac{ 1}{|\CG|}) \le 2 \tmix(H)$ and one can argue that mixing on $\Hh\wr\CG$ requires mixing on a single lamp graph $\Hh_w$ for a fixed $w\in \CG$. Thus the lower bound remains valid.

The cover time of $G$
is already a lower bound for the $0-1$ lamps case by \cite{PR04}, hence also for general lamps,
but, for completeness, we adjust the proof  in \cite[Theorem 19.2]{LPW08} to our setting. By Lemma \ref{lem::opt_sst_lamplighter} we can estimate the separation distance on $\Hh\wr\CG$ as
\be \ba \label{eq::stgecov}  s^\diamond_{(\un f, x)}(t)&\ge \Pv_{(\un f, x)}\left[\tau^\diamond > t\right] \\
&\ge \Pv_{(\un f, x)}\left[\exists w \in \CG: \tau_H(w)> L_w(t) \right]\\
& \ge \Pv_{(\un f, x)}\left[\exists w\in \CG: L_w(t)=0 \right] = \Pv_{(\un f, x)}\left[\tau_{\text{cov}}>t\right].  \ea\ee
Now, using the submultiplicativity of $d(t)$ in \eqref{eq::tv_submultiplicative} and the relation of the separation distance and the total variation distance in \eqref{eq::sepletv}, we have that at time $8\tmix(H\wr G, 1/4)$:
\[  s^\diamond_{(\un f, x)} \left(8\tmix(H\wr G, \tfrac14)\right) \le 4 d^\diamond \left(4\tmix(H\wr G, \tfrac14)\right) \le 4 \frac{2^4}{4^4}\le \frac14 \]
Combining with \eqref{eq::stgecov} yields that for every starting state we have
\[ \Pv_{(\un f, x)}\left[\tau_{\text{cov}} > 8\tmix(H\wr G, 1/4) \right] \le 1/4. \]
Thus, run the chain in blocks of $8\tmix(H\wr G, 1/4)$ and conclude that in each block it covers with probability at least $3/4$. Thus, the cover time is dominated by $8\tmix(H\wr G, 1/4)$ times a geometric random variable with success probability $3/4$, so we have
\[ \Ev_{(\un f, x)}\left[\tau_\text{cov}\right] \le 11 \tmix(\Hh\wr \CG, 1/4)  . \]
Maximizing the left hand side over all possible starting states yields $\tcov(\CG) \le 11 \tmix(\Hh\wr \CG, 1/4)$, finishing the proof.
\subsubsection{Proof of the lower bound of Theorem \ref{thm::main_mixing}, without assumption {\bf(A)}}
Now we turn to the general case and first show that $c\  \trel(H) |G| \log |G|$ is a lower bound. No laziness assumption on the chain on $H$ is needed to get this bound. We will use a distinguishing function method.
Namely, take an eigenfunction $\phi_2$ of the transition matrix  $Q$ on $H$ corresponding to the second eigenvalue $\la_H$.
Then let us define $\psi: H \wr G \to \mathbb C$:
\be \psi((\un f, x)):= \sum_{v\in G} \phi_{2} (f_v). \ee\label{eq::psi_def}
One can always normalize such that
\[ \Ev_{\pid}(\psi)= \sum_{v \in G} \Ev_\pi[\phi_{2}]=0 \quad \Vv_{\pid}(\psi)=\sum_{v \in G}\Vv_\pi(\phi_{2}) = |G|\cdot 1\]
This normalization has two useful consequences: First, by Chebyshev's inequality, the set $A=\{\psi<2 |G|^{1/2}\}$ has measure at least $3/4$ under stationarity. Second, $\phi_2(g_0):=\max_{g\in H} \phi_2(g) > 1$, otherwise the variance would be less than $1$. We aim to show that the set $A$ has measure less then $1/2$ at time $c \trel(H) |G|\log |G|$ and then we are done by using the following characterization of the total variation distance, see \cite{AF02, LPW08}:
\[ \| \nu-\mu \|_{\TV}= \sup_{A\subset \Omega} \{ \nu(A) - \mu(A)\}. \]

Let us start all the lamp graphs from $g_0\in H$ where the maximum is attained for $\phi_2$. Then we can condition on the local time sequence and use the eigenvalue property of $\phi_2$ to obtain
 \be \ba\label{eq::psi_est_1} \Ev_{(\un g_0, x)}\left[\psi((\un F_t, X_t))\right]& = \Ev \left[ \Ev\left[\sum_{w\in G}\phi_2(F_w(t))| \left(L_v(t)\right)_v \right] \right] \\
 & = \phi_2(g_0) \Ev_x \left[ \sum_{w\in G} \la_H^{L_w(t)} \right].
 \ea \ee
 Since $\sum_v L_v(t) = 2t$, we can apply Jensen's inequality on the function $y\mapsto\la_H^y$ to get a lower bound on the expectation:
 \be\label{eq::general_exp_lower}  \Ev_x \left[ \sum_{w\in G} \la_H^{L_w(t)}\right] \ge |G| \la_H^{\tfrac{2t}{|G|}}= |G| \left(1- \frac{1}{\trel(H)} \right)^{\tfrac{2t}{|G|}}.\ee
By giving a lower bound on the right hand side we must assume here that $\la_H>0$, or equivalently $\trel(H)>C>1$. Thus, first we handle the other case, i.e. when $\trel(H)<2$. Then the lower bound we are about to show is of order $|G| \log |G|$ which is always at most the order of $\tcov(G)$, due to a result by Feige \cite{Feige95} stating that for simple random walk on any connected graph $G$, $\tcov(G)\ge (1+o(1))|G| \log |G|$.

When $\trel(H)>2$,  we can use that $1-x > e^{-1.5 x } $ when $0<x < 0.5$ to get a lower bound on the right hand side of \eqref{eq::general_exp_lower}. Then set $t=c \trel(H) |G| \log |G|$ turning the estimate in \eqref{eq::psi_est_1} into
\[  \Ev_{(\un g_0, x)}\left[\psi((\un F_t, X_t))\right] \ge |G|^{1-3c} \phi_2(g_0).\]
We can easily upper bound the conditional variance as follows:
\[ \Vv \left[\psi_t | (L_v(t))_{v\in G}\right] \le \sum_{w\in G} \Ev_{g_0}\left[\phi_2^2(F_w(t))| L_w(t)\right] \le |G| \phi_2^2(g_0).\]
Now, let us estimate the measure of set $A$ at time $t$ by using the lower bound on the expectation:
 \[ \Pv_{(\un g_0, x)} \left[\psi_t \le 2 |G|^{1/2} \right] \le \Pv_{(\un g_0, x)} \left[|\psi_t - \Ev(\psi_t)| \ge \phi_2(g_0)|G|^{1-3c} - 2|G|^{1/2} \right] \]
 Now we use that $\phi_2(g_0)>1$ and if $c<1/6$ then on the right hand side, the term $\phi_2(g_0)|G|^{1-3c}$ dominates, so for $|G|$ large enough we can drop the negative term and compensate it with a multiplicative factor of $1/2$, say.
 Thus, condition on the local time sequence first and see that for any sequence $\left(L_v(t)\right)_{v\in G}$ Chebyshev's inequality yields:
 \[ \Pv_{(\un g_0, x)} \left[\psi_t \in A \big| \left(L_v(t)\right)_{v\in G} \right]\le  \frac{\Vv\left[\psi_t | (L_v(t))_{v\in G}\right]}{ 1/4 \phi_2^2(g_0)|G|^{2-6c}}\]
 Combining this with the estimate on the conditional variance above yields that
 \[ \Pv_{(\un g_0, x)}\left[\psi_t \in A| \left(L_v(t)\right)_v \right] \le \frac{4}{|G|^{1-6c}}. \]
 This bound is independent of the local time sequence, so the law of total probability says we have the same upper bound without conditioning on the local times.
 Now setting $c<1/6$ an $|G|$ large enough we see that the right hand side can be made smaller than $1/2$, finishing the proof.

To see that the cover time is a lower bound in the general case, couple the chain on $H\wr G$ to $\Z_2\wr G$, i.e. jump to stationary distribution on $H_v$ once the walker on the base hits vertex $v$ and use \cite{PR04} or \cite{LPW08} to see that $\tcov(G)\le \tmix(Z_2\wr G)\le \tmix(H\wr G)$.

Next we show that $c |G| \tmix(H)$ is a lower bound if the chain on $H$ is lazy.

Let us start with a definition for general Markov chain $X$ on $\Omega$
\[ \tstop(G):= \max_{x \in \Omega} \min\left\{ \Ev[\tau]; \tau \mbox{ stopping time s.t. } \Pv_x[X_\tau = y]= \pi(y)\  \forall y\in \Omega \right\}. \] We call a stopping time mean-optimal if $\Ev[\tau]=\tstop(G)$. Lov\'asz and Winkler \cite{LW95} show that optimal stopping rules always exist for irreducible Markov chains.
We aim to show that
\[ \frac{1}{2}|G| \cdot \tstop (H) \le \tstop (H\wr G ).\]
 Take a mean optimal stopping time $\tau^*$ on $H\wr G$ reaching minimal expectation, i.e. $\Ev_{(\un f^*,x^*)} [\tau^*]=\tstop(H\wr G)$ for some $(\un f^*, x^*)\in H\wr G$ and $\Ev_{(\un f,x)} [\tau^*]\le\tstop(H\wr G)$ for $(\un f, x)\neq (\un f^*, x^*)$.

 We use this $\tau^*$ to define a stopping rule $\tau_H(v)$ on $H_v$, for every $v \in G$. Namely, do the following: look at a coordinate $v\in G$ and at the chain restricted to the lamp graph $H_v$, i.e. only the moves which are done on the coordinate $H_v$. Then, stop the chain on $H_v$  when $\tau^*$ stops on the whole $H\wr G$.

 Start the chain from any $(\un f_0, x_0)$.  Since $\sum_{v\in G} L_v(t)=2 t$, we have
 \[ \sum_{v\in G} \Ev_{f_v(0)}[\tau_H(v)]= \Ev_{(\un f_0, x_0)}\left[ \sum_{v\in G} L_v(\tau^*) \right]= 2 \Ev_{(\un f_0, x_0)}[\tau^*]. \]
Take the vertex $w\in G$ (which can depend on $x_0$), which minimizes the expectation $\Ev_{f_v(0)}[\tau_H(w)]$. Clearly for this vertex the expected value must be less than the average:
 \[ \Ev_{f_0}[\tau_H] \le \frac{2}{|G|} \Ev_{(\un f_0, x_0)}[\tau^*] \]
The left hand side is at least as large as what a mean-optimal stopping rule on $H$ can achieve, and the right hand side is at most $\frac{2}{|G|}\tstop(H\wr G)$. Thus we arrive at
\[ \frac{1}{2} |G| \tstop(H) \le \tstop(H\wr G). \]

In the last step we use the equivalence from the paper \cite[Corollary 2.5]{PS12} stating that $\tstop$ and $\tmix$ are equivalent up to universal constants for lazy reversible chains and get that

\[ c_1 |G| \tmix(H) \le \tmix(H\wr G). \]


\section{Further directions}
The next step of understanding generalized lamplighters walks might be to investigate which properties on $G$ and $H$ are needed to exhibit cutoff (for a definition see \cite{AF02, LPW08}), or to determine the mixing time in the uniform metric.

For $\Z_2 \wr \CG$, already \cite{HJ97} implies a total variation cutoff with threshold $\tfrac{1}{2} \tcov(K_n)$ for $G$ being the complete graph and that there is no cutoff if $G$ is a cycle of length $n$. The results of \cite{PR04} include a proof of total variation cutoff for $\Z_2 \wr \Z_n^2$ with threshold $\tcov(\Z_n^2)$.  The results in \cite{MP11} also includes cutoff at $1/2 \tcov(G_n)$, with some uniform local transience assumptions on $G_n$. Further, Levi \cite{NL06} proved that the wreath product of two complete graphs $K_{n^{\la}}\wr K_n$, $0\le \la \le 1$ exhibits a cutoff at $(1+\la)/2 n \log n$.

For the mixing time in the uniform metric, we know \cite[Theorem 1.4]{PR04} that if $\CG$ is a regular graph such that $\thit(\CG) \leq K|\CG|$, then there exists constants $c,C$ depending only on $K$ such that
\be \label{eqn::tu_mix} c|\CG|(\trel(\CG) + \log|\CG|) \leq \tu(\Z_2\wr \CG) \leq C|\CG|(\tmix(\CG) + \log|\CG|).\ee
These bounds fail to match in general.  For example, for the hypercube $\Z_2^d$, $\trel(\Z_2^d) = \Theta(d)$
 \cite[Example 12.15]{LPW08} while $\tmix(\Z_2^d) = \Theta(d \log d)$ \cite[Theorem 18.3]{LPW08}. Then \cite{KMP12} showed that the lower bound is sharp in \eqref{eqn::tu_mix} under conditions which are satisfied by the $d(n)$ dimension tori $\CG_n= \Z_n^{d(n)}$ for arbitrary chosen $n$ and $d(n)$.

\section{Acknowledgement}
We thank Elisabetta Candellero, G\'abor Pete and Tim Kam Wong for useful comments and Louigi Addario-Berry, Perla Sousi and Peter Winkler for a useful discussion of halting states.

\bibliographystyle{plain}
\bibliography{lamplighter}

\end{document}